\documentclass[11pt]{amsart}
\usepackage[all]{xy}
\usepackage{amsfonts,amsmath,amssymb,amscd}
\usepackage{color}
\usepackage[pagebackref,colorlinks, linkcolor=blue, citecolor=red]{hyperref}
\usepackage{caption} 
\usepackage{hyperref}
\usepackage{stmaryrd}
\usepackage{mathtools}

\usepackage[top=1in, bottom=1in, left=.94in, right=.94in]{geometry}
\usepackage{graphicx}

\numberwithin{equation}{section}
\newtheorem{theorem}{Theorem}[section]
\newtheorem{proposition}[theorem]{Proposition}

\newtheorem{coro}[theorem]{Corollary}
\newtheorem{definition}[theorem]{Definition}

\newtheorem{remark}[theorem]{Remark}
\newtheorem{example}[theorem]{Example}
\newtheorem{question}[theorem]{Question}

\title[Reduced Type of certain Numerical Semigroup Rings]{Reduced Type of certain Numerical Semigroup Rings}

\author{Om Prakash Bhardwaj}
\address{IIT Bombay, Powai, Maharashtra-400076, India}
\email{om.prakash@math.iitb.ac.in}
\thanks{}


\begin{document}
\begin{abstract}
For a reduced one-dimensional complete local $k$-algebra $R$, Huneke et al. (Res. Math. Sci., 8(4), paper no. 60, 2021) introduced an important invariant, the reduced type. In this article, we study the extremal behavior of reduced type of some special numerical semigroup rings. For a numerical semigroup ring, the behavior of reduced type can be studied by analyzing the set of pseudo-Frobenius elements of the numerical semigroup.  We give complete descriptions of pseudo-Frobenius elements of Bresinsky's numerical semigroups and duplication of numerical semigroups. Further, we explore the extremal behavior of reduced type for the associated semigroup rings.
\end{abstract}

\maketitle

\section{Introduction}
Let $k$ be an algebraically closed field of characteristic zero and $(R,\mathbf{m},k)$ be an equicharacteristic reduced one-dimensional complete local $k$-algebra. Let $x$ be a minimal reduction in the maximal ideal $\mathbf{m}$ and $\mathfrak{C} = (R :_{K} \bar{R})$ be the conductor ideal of $R$, where $K$ is the quotient field of $R$. To explore the Berger's conjecture \cite{berger} (which states that the universally finite module of differentials, $\Omega_R$, is torsion-free if and only if $R$ is regular), Huneke et al. in \cite{huneke} introduced an important invariant, reduced type $s(R)$ defined as $s(R):= \mathrm{dim}_{k}\frac{\mathfrak{C}+xR}{xR}$. Also, this invariant has been used by the authors in \cite{mukundan2} to study the non-zero torsions in the module of differentials, $\Omega_R$. The invariant $s(R)$ is related to the invariant Cohen-Macaulay type of $R$, denoted by $\mathrm{type}(R)$. Since $\frac{\mathfrak{C}+xR}{xR} \subseteq \frac{xR : \mathbf{m}}{xR}$, and the dimension of $\frac{xR : \mathbf{m}}{xR}$ as a $k$-vector space is precisely the Cohen-Macaulay type of $R$, one has $ 1 \leq s(R) \leq \mathrm{type}(R).$ Following this observation, the authors in \cite{mukundan} defined that $R$ has minimal reduced type if $s(R)=1$, and $R$ has maximal reduced type if $s(R) = \mathrm{type}(R).$ Thus,  $R$ has both maximal and minimal reduced type if and only if $R$ is Gorenstein. For non-Gorenstein rings, it is natural to ask the following question:
\begin{question}[{\cite[Question 2.5]{mukundan}}]\label{question}
 Classify, when $R$ has maximal or minimal reduced type?
\end{question} 

Let $S$ be a numerical semigroup, i.e., a finitely complemented submonoid of non-negative integers $\mathbb{N}$. Let $n_1< n_2 < \ldots < n_e$ be the minimal generators of $S$. Define a map $\phi: k \llbracket X_1, \ldots, X_e \rrbracket \rightarrow k \llbracket t \rrbracket$ such that $X_i = t^{n_i}$ for $1 \leq i \leq e$. Then the one-dimensional local $k$-algebra $R:= k\llbracket t^{n_1}, \ldots, t^{n_e} \rrbracket \cong \frac{k\llbracket X_1, \ldots, X_e \rrbracket}{\mathrm{ker(\phi)}}$ is called the numerical semigroup ring associated to $S$, denoted by $k\llbracket S \rrbracket$. In this article, $R$ will be a semigroup ring associated with a numerical semigroup $S$. There are nice relations between the properties of the $k$-algebra $R:= k\llbracket S \rrbracket$ and the numerical semigroup $S$. Kunz \cite{kunz}, proved that $R$ is Gorenstein if and only if its value semigroup $S$ is symmetric (i.e., $s \notin S$ implies $\mathrm{F}(S)-x \in S$), where $\mathrm{F}(S)$ is the Frobenius number of the numerical semigroup $S$. Barucci and Fr\"{o}berg \cite{barucci-froberg}, showed that $R$ is almost Gorenstein if and only if the numerical semigroup $S$ is almost symmetric. Also, it is well known that the Cohen-Macaulay type of $R$ is equal to the cardinality of the set of pseudo-Frobenius numbers of the numerical semigroup $S$. In \cite[Theorem 2.13]{mukundan}, the authors gave a characterization of the reduced type of a numerical semigroup ring. They prove that $s( k\llbracket  S \rrbracket)$ is equal to the cardinality of the set $[ \mathrm{F}(S)-m(S)+1,\mathrm{F}(S)] \setminus S$, where $m(S)$ is the smallest element of $S$, called the multiplicity of $S$. Using this characterization, they study Question \ref{question} for the almost Gorenstein numerical semigroup rings, numerical semigroup rings of minimal multiplicity, far-flung Gorenstein numerical semigroup rings, and the numerical semigroup rings associated with the gluing of numerical semigroups. In this article, we study Question \ref{question} for other important classes of numerical semigroup rings.

Now, we summarize the contents of the paper. Section 2 recalls some definitions and results about numerical semigroups and reduced type. Let $p\geq 2$, $n_0,d,s \geq 1$, and $S = \langle n_0, sn_0+d , \ldots, sn_0+pd \rangle$ be a numerical semigroup minimally generated by the generalized arithmetic sequence $n_0, sn_0+d , \ldots, sn_0+pd.$ In section 3, we consider three well studied classes of numerical semigroups: (i) numerical semigroups generated by generalized arithmetic sequences, (ii) numerical semigroups generated by the integers defined in Backelin's curves, and (iii) numerical semigroups generated by the integers defined in Bresinsky's curves.  The description of pseudo-Frobenius numbers of numerical semigroups generated by generalized arithmetic sequences given in \cite{matthews}, using this we study the reduced type of these numerical semigroup rings.
Next, we study the extremal behavior of reduced type of the numerical semigroup rings defined by the integers of the parametrizations of Bresinsky's and Backelin's monomial curves. In \cite{bresinsky1}, Bresinsky gave an example of monomial curves in the affine $4$-space $\mathbb{A}^4$ defined by the integers $n_1 = 2h(2h+1)$, $n_2 = (2h-1)(2h+1)$, $n_3 = 2h(2h+1)+(2h-1)$, and $n_4 = (2h-1)2h$, where $h \geq 2$. In \cite{froberg}, Fr\"{o}berg et al. communicated a family of semigroups that was introduced by Backelin, which is defined by the integers  $n_1 = r(3n+2)+3$, $n_2 = r(3n+2)+6$, $n_3 = r(3n+2)+3n+4$, $r(3n+2)+3n+5$, where $n \geq 2$, and $r\geq 3n+2$. A description of pseudo-Frobenius numbers of numerical semigroups generated by the integers defined in Backelin's curves is given in \cite[Proposition 1.8]{enescu}. In Theorem \ref{pfbresinsky}, we give a complete description of the pseudo-Frobenius numbers of numerical semigroups generated by the integers defined in Bresinsky's curves. Using these descriptions of pseudo-Frobenius numbers, we prove:
Let $S = \langle n_1,n_2,n_3,n_4 \rangle$ be a numerical semigroup generated by $n_1,n_2,n_3,n_4$ as defined in Bresinsky's curves, and $S' = \langle n_1,n_2,n_3,n_4 \rangle$ be a numerical semigroup generated by the integers $n_1,n_2,n_3,n_4$  defined in Backelin's curves. Then $k\llbracket  S \rrbracket$ and $k\llbracket  S' \rrbracket$ never have maximal or minimal reduced type.

In section 4, we study the extremal behavior of the reduced type of semigroup rings associated with a special kind of gluing of numerical semigroups. Delorme \cite{delorme}  gave the first construction of gluing to characterize the complete intersection numerical semigroup rings. Recall the definition of gluing of numerical semigroups by Rosales from \cite{rosales97}. Let $S_1$ and $S_2$ be two numerical semigroups minimally generated by $n_1,\ldots,n_r$ and $n_{r+1},\ldots,n_e$ respectively. Let $\mu \in S_1 \setminus \{ n_1,\ldots,n_r \}$ and $\lambda \in S_2 \setminus \{ n_{r+1}, \ldots, n_e \}$ be such that $\gcd(\lambda,\mu)=1.$ We say that $S = \langle \lambda n_1, \ldots, \lambda n_r, \mu n_{r+1}, \ldots, \mu n_{e} \rangle$ is a gluing of $S_1$ and $S_2,$ denoted by $S_1+_{\lambda \mu} S_2$. In \cite[Theorem 3.34]{mukundan}, the authors prove that if a numerical semigroup $S$ is a gluing of numerical semigroups $S_1$ and $S_2$ then $s(k\llbracket S \rrbracket) \leq s(k\llbracket S_1 \rrbracket) s(k\llbracket S_1 \rrbracket)$, and consequently prove that if $k\llbracket S_1 \rrbracket$ and $k\llbracket S_2 \rrbracket$ have minimal reduced type so has $k\llbracket S \rrbracket$. However, such a nice analog for maximal reduced type does not hold. In Theorem \ref{gluingmaxredtype}, we give a sufficient condition for this analog for a special kind of gluing of numerical semigroups, and as a consequence, we obtain, in Corollary \ref{maxniceext}, that a numerical semigroup ring associated with a nice extension of a numerical semigroup $S$ is of maximal reduced type if and only if $k \llbracket S \rrbracket$ has maximal reduced type.

In section 5, we study the extremal behavior of the reduced type of the semigroup ring associated with the duplication of a numerical semigroup. The motivation of the duplication comes from the idealization of the a module over a commutative ring with unity, introduced by Nagata (see \cite[page 2]{nagata}). Let $S$ be a numerical semigroup and $E \subseteq S$ be an ideal of $S$, i.e., $S + E \subseteq E$. For an odd integer $d \in S$ and an ideal $E \subseteq S$, the authors in   \cite{d'anna}, defined the numerical duplication of $S$ with respect to $E$ and $d$ as $$ S \bowtie^d E = 2 \cdot S \cup (2\cdot E + d),$$  
\noindent
where $2 \cdot S = \{2s \mid s \in S\}$ and $2 \cdot E + d = \{2e + d \mid e \in E\}$. In Theorem \ref{pfduplication}, we give a complete description of the set of pseudo-Frobenius elements of the duplication of a numerical semigroup. Using this description, in the remaining section (Theorem \ref{minreddup}, Proposition \ref{maxreddup}, Proposition \ref{maxreddup2}),  we explore when the duplication $k \llbracket S \bowtie^d E \rrbracket$ has minimum/maximum reduced type. Thus, we observe the following three important points with the help of the construction of the duplication:
\begin{enumerate}
\item[(i)] For a fixed $r > 0$, there are infinitely many numerical semigroups with type $r$.
\item[(ii)] For a fixed $r > 0$, there are infinitely many numerical semigroups of minimal reduced type with type $r$. 
\item[(iii)] For a fixed $r > 0$, there are infinitely many numerical semigroups of maximal reduced type with type $r$.
\end{enumerate}

\section{Preliminaries}
Let $\mathbb{Z}$ and $\mathbb{N}$ denote the sets of integers and non-negative integers, respectively. For the basics about numerical semigroups, one can refer to \cite{numerical}. Now, we recall some useful definitions and results.

\begin{definition}
Let $S$ be a submonoid of $\mathbb{N}$ such that $\mathbb{N} \setminus S$ is finite, then $S$ is called a numerical semigroup. Equivalently, there exist $n_1,\ldots,n_e \in \mathbb{N}$ such that $\mathrm{gcd}(n_1,\ldots,n_e) = 1$ and 
\[
S =  \langle n_1, \ldots, n_e \rangle = \left\lbrace \sum_{i=1}^{e} \lambda_i n_i \mid \lambda_i \in \mathbb{N}, \forall i \right\rbrace.
\]
\end{definition}

It is known that every numerical semigroup has a unique minimal generating set; the cardinality of this set is known as the embedding dimension of $S$. The smallest non-zero element of $S$ is called the multiplicity of $S$, denoted by $m(S)$, and the largest element in $\mathbb{N}\setminus S$ is called the Frobenius number of $S$, and it is denoted by $\mathrm{F}(S)$. 

\begin{definition}
Let $S$ be a numerical semigroup. An element $f \in \mathbb{Z} \setminus S$ is called a pseudo-Frobenius number if $f + s \in S$ for all $s \in S\setminus \{0\}$. The set of pseudo-Frobenius numbers of $S$ is denoted by $\mathrm{PF}(S)$. Note that $\mathrm{F}(S) \in \mathrm{PF}(S)$ and $\mathrm{F}(S)$ is the maximum element of $\mathrm{PF}(S)$. We denote the set $\mathrm{PF}(S) \setminus \{\mathrm{F}(S)\}$ of pseudo-Frobenius number other than Frobenius number by $\mathrm{PF}'(S)$.
\end{definition}

Let $S$ be a numerical semigroup minimally generated by $n_1, \ldots, n_e$. The local $k$-algebra $k\llbracket S \rrbracket := k\llbracket t^{n_1}, \ldots, t^{n_e} \rrbracket$ is called the numerical semigroup ring associated to the numerical semigroup  $S$. The ring $ k\llbracket t^{n_1}, \ldots, t^{n_e} \rrbracket \cong R = \frac{k\llbracket X_1, \ldots, X_e \rrbracket}{I}$ for some $I \subseteq (X_1, \ldots, X_e)^2.$ 

\begin{definition}
The cardinality of the set of pseudo-Frobenius numbers is known as the type of the numerical semigroup $S$, which is equal to the Cohen-Macaulay type of the numerical semigroup ring $k\llbracket S \rrbracket$, denoted by $\mathrm{type}(k\llbracket S \rrbracket)$.
\end{definition} 

\begin{definition}[{\cite{huneke}}]
Let $S$ be a numerical semigroup minimally generated by $n_1, < n_2< \ldots < n_e$. Let $R = k\llbracket  S \rrbracket$ be the numerical semigroup ring associated with $S$. Let $x_i$ be the images of $X_i$ modulo $I$, the reduced type of $R$, denoted by $s(R)$ is defined as
$$
s(R) = \mathrm{dim}_{k}\frac{\mathfrak{C}+x_1R}{x_1R},
$$
where $\mathfrak{C} = (R :_{K} \bar{R})$ is the conductor ideal of $R$, and $K$ is the quotient field of $R$. The reduced type $s(R)$ satisfies the inequality $ 1 \leq s(R) \leq \mathrm{type}(R).$ 
\end{definition}

 Following this observation, the authors in \cite{mukundan} defined the notion of maximal and minimal reduced type as follows:

\begin{definition}
Let $R = k\llbracket  S \rrbracket$ be a numerical semigroup ring. Then $R$ has maximal reduced type if $s(S) = \mathrm{type}(S)$. Analogously, $R$ has minimal reduced type if $s(S) = 1$.
\end{definition}

In \cite[Theorem 2.13]{mukundan}, the authors proved that the reduced type $s(R)$ of a numerical semigroup ring $R = k\llbracket  S \rrbracket$ is equal to the cardinality of the set $[ \mathrm{F}(S)-m(S)+1,\mathrm{F}(S)] \setminus S$. Note that $[ \mathrm{F}(S)-m(S)+1,\mathrm{F}(S)] \setminus S \subseteq \mathrm{PF}(S)$. Using this, one can deduce the following criterion for a numerical semigroup ring for having maximal/minimal reduced type as:  

\begin{proposition}\label{maxred}
Let $S$ be a numerical semigroup. The semigroup ring $k\llbracket  S \rrbracket$ has maximal reduced type if and only if $\mathrm{min}~\mathrm{PF}(S) \geq \mathrm{F}(S)-m(S)+1$. 
\end{proposition}

\begin{proposition}\label{minred}
Let $S$ be a numerical semigroup. The semigroup ring $k\llbracket  S \rrbracket$ has minimal reduced type if and only if $\mathrm{max}~\mathrm{PF}'(S) < \mathrm{F}(S)-m(S)+1$.
\end{proposition}

We will use the above results for a semigroup ring having maximal and minimal reduced type and use them throughout the article without repeated reference.

\section{Some interesting classes of numerical semigroups}

\subsection{Numerical semigroups generated by generalized arithmetic sequences:} Numerical semigroups with embedding dimension two always have type one. Thus, the semigroup rings associated with them always have both maximum and minimum reduced type. In this section, we will consider the numerical semigroups with embedding dimension greater than or equal to three and minimally generated by generalized arithmetic sequences. Let $p\geq 2$, $n_0,d,s \geq 1$, and $S = \langle n_0, sn_0+d , \ldots, sn_0+pd \rangle$ be a numerical semigroup minimally generated by the generalized arithmetic sequence $n_0, sn_0+d , \ldots, sn_0+pd.$  For $s = 1$, a generalized arithmetic sequence is an arithmetic sequence. These numerical semigroups have been studied in \cite{matthews}.

\begin{theorem}\label{pfgas}
Let $S = \langle n_0, n_1, \ldots, n_p \rangle$ be a numerical semigroup minimally generated by a generalized arithmetic sequence, where $n_i = sa + id$ for $1\leq i \leq p$. Write $n_0 = ap + b, 0 \leq b < p$, then we have the following:
\begin{enumerate}
\item[(i)] If $b = 0$, then
$\mathrm{PF}(S) = \{an_p -n_0 - (p-i)d \mid 1 \leq i \leq p-1\}.$
\item[(ii)] If $b = 1$, then
 $\mathrm{PF}(S) = \{an_p - n_0 - (p-i)d \mid 1 \leq i \leq p\}.$
 \item[(iii)] If $b \neq 0,1$, then
 $\mathrm{PF}(S) = \{an_p + id \mid 1 \leq i \leq b-1\}.$
\end{enumerate}
\end{theorem}
\begin{proof}
 Define $A_a = \{asn_0 + ld \mid (a-1)p + 1 \leq l \leq ap\}$. The theorem can be easily deduced from \cite[Theorem 3.1]{numata2}, where it is given that,
 \begin{enumerate}
 \item[(i)] If $b = 0$, then
$\mathrm{PF}(S) = \{\omega - n_0 \mid \omega \in A_a \setminus \{asn_0 + apd\}\},$
\item[(i)] If $b = 1$, then
$\mathrm{PF}(S) = \{\omega - n_0 \mid \omega \in A_a \},$
\item[(i)] If $b \neq 0,1$, then
$\mathrm{PF}(S) = \{(as+1)n_0+ld-n_0 \mid ap+1 \leq l \leq ap+b-1\}.$ \qedhere
 \end{enumerate}
\end{proof}

\begin{proposition}\label{maxgas}
Let $S = \langle n_0, n_1, \ldots, n_p \rangle$ be a numerical semigroup minimally generated by a generalized arithmetic sequence, where $n_i = sn_0 + id$ for $1\leq i \leq p$. Write $n_0 = ap + b, 0 \leq b < p$, then we have the following:
\begin{enumerate}
\item[(i)] If $b =0$ then $k\llbracket  S \rrbracket$ has maximal reduced type if and only if $p-2 \leq \frac{n_0-1}{d}.$
\item[(ii)] If $b =1$ then $k\llbracket  S \rrbracket$ has maximal reduced type if and only if $p-1 \leq \frac{n_0-1}{d}.$
\item[(iii)] If $b \neq 0,1$ then $k\llbracket  S \rrbracket$ has maximal reduced type if and only if $b-2 \leq \frac{n_0-1}{d}.$
\end{enumerate}
\end{proposition}
\begin{proof}
If $b = 0$, then by Theorem \ref{pfgas} we have $\mathrm{F}(S) = a(n_p-1) -d$ and $\mathrm{min}~ \mathrm{PF}(S) = a(n_p-1)-(p-1)d$. Therefore,  $k\llbracket  S \rrbracket$ has maximal reduced type if and only if 
\begin{align*}
 a(n_p-1)-(p-1)d \geq a(n_p-1) -d -n_0 + 1  \iff (n_0 - 1)d \geq (p-2)d.
\end{align*}
If $b = 1$, then by Theorem \ref{pfgas} we have $\mathrm{F}(S) = a(n_p-1)$ and $\mathrm{min}~ \mathrm{PF}(S) = a(n_p-1)-(p-1)d$. Therefore,  $k\llbracket  S \rrbracket$ has maximal reduced type if and only if 
\begin{align*}
 a(n_p-1)-(p-1)d \geq a(n_p-1) -n_0 + 1  \iff n_0 - 1 \geq (p-1)d.
 \end{align*}
If $b \neq 0,1$, then by Theorem \ref{pfgas} we have $\mathrm{F}(S) = an_p + (b-1)d$ and $\mathrm{min}~ \mathrm{PF}(S) = an_p + d$. Therefore,  $k\llbracket  S \rrbracket$ has maximal reduced type if and only if 
$$
 an_p + d \geq an_p + (b-1)d - n_0 + 1  \iff n_0 - 1 \geq (b-2)d. \qedhere 
 $$ 
\end{proof}

\begin{proposition}\label{mingas}
Let $S = \langle n_0, n_1, \ldots, n_p \rangle$ be a numerical semigroup minimally generated by a generalized arithmetic sequence, where $n_i = sn_0 + id$ for $1\leq i \leq p$. Write $n_0 = ap + b, 0 \leq b < p$, then we have the following:
\begin{enumerate}
\item[(i)] If $b = 2$ then $k\llbracket  S \rrbracket$ always has minimal reduced type.
\item[(ii)] If $b \neq 2$ then $k\llbracket  S \rrbracket$ has minimal reduced type if and only if $n_0 < d-1$.
\end{enumerate}

\end{proposition}
\begin{proof}
If $b = 2$, note from Theorem \ref{pfgas} that $\vert \mathrm{PF}(S) \vert = \mathrm{type}(k\llbracket  S \rrbracket) = 1$. Thus, we get $1 \leq s(k\llbracket  S \rrbracket) \leq \mathrm{type}(k\llbracket  S \rrbracket) = 1$. Therefore $s(k\llbracket  S \rrbracket) = 1$, and hence $k\llbracket  S \rrbracket$ has minimal reduced type.
If $b \neq 2$, from Theorem \ref{pfgas}, we see that the difference between $\mathrm{F}(S)$ and $\mathrm{max}~\mathrm{PF}'(S)$ is always $d$. Thus $\mathrm{max}~\mathrm{PF}'(S) < \mathrm{F}(S)-n_0+1$ if and only is $n_0 < d+1$. Hence, $k\llbracket  S \rrbracket$ has minimal reduced type.
\end{proof}

\begin{example}
\rm{(1) Let $s = 3$, $n_0 = 5$, $d = 7$, and $S = \langle 5, 22, 29, 36, 43 \rangle$. Since $n_0 < d-1$, $k\llbracket  S \rrbracket$ has minimal reduced type. Also, here $p = 4$, and thus $b = 1$. Since $p-1 > \frac{n_0-1}{d}$, $k\llbracket  S \rrbracket$ does not have maximal reduced type. \\
(2) Let $s = 2$, $n_0 = 12$, $d = 5$, and $S = \langle 12, 29, 34, 39, 44, 49, 54, 59, 64 \rangle$. Since $n_0 > d-1$, $k\llbracket  S \rrbracket$ does not have minimal reduced type. Also, here $p = 8$, and thus $b = 4$. Since $b-2 \leq \frac{n_0-1}{d}$, $k\llbracket  S \rrbracket$ has maximal reduced type.\\
(3) Let $s = 5$, $n_0 = 7$, $d = 11$, and $S = \langle 7, 46, 57, 68, 79, 90 \rangle$. Here $p = 5$, thus $b = 2$, and hence $k\llbracket  S \rrbracket$ has both maximal and minimal reduced type.\\
(4) Let $s = 1$, $n_0 = 17$, $d = 7$, and $S = \langle 17, 24, 31, 38, 45 \rangle$. Here $p = 4$, and thus $b = 1$. Since $n_0 > d-1$, $k\llbracket  S \rrbracket$ does not have minimal reduced type. Also since $p-1 > \frac{n_0-1}{d}$, $k\llbracket  S \rrbracket$ does not have maximal reduced type.}
\end{example}

\subsection{Backelin's Curves}

Let $n \geq 2$, $r\geq 3n+2$, $n_1 = r(3n+2)+3$, $n_2 = r(3n+2)+6$, $n_3 = r(3n+2)+3n+4$, $r(3n+2)+3n+5$. The monomial curves defined by $n_1,n_2,n_3,n_4$ are known as Backelin curves. Let $S$ be a numerical semigroup generated by $n_1, n_2, n_3$, and $n_4$. This family of semigroups was introduced by Backelin and communicated for the first time by Fr\"{o}berg et al. in \cite{froberg}. They proved that $\mathrm{type}(k\llbracket  S \rrbracket) \geq 2n+2$ and stated that $\mathrm{type}(k\llbracket  S \rrbracket) = 2n+3$. However, the value of the type of $k\llbracket  S \rrbracket$ they stated was incorrect. The exact value of the type has been provided in \cite{enescu} as $\mathrm{type}(k\llbracket  S \rrbracket) = 3n+2$, and also all the elements of $\mathrm{PF}(S)$ have been computed.

\begin{proposition}[{\cite[Proposition 1.8]{enescu}}]\label{pfbackelin}
Let $S = \langle n_1,n_2,n_3,n_4 \rangle$ be a numerical semigroup generated by the integers $n_1,n_2,n_3,n_4$ defined in Backelin's curves. Then 
\begin{align*}
\mathrm{PF}(S) = \left\lbrace 
\begin{array}{c}
(n-k)n_1 + (3k-2)n_3 -n_4, ~\text{for}~ 2 \leq k \leq n\\[1mm]
(r-(n+k)+3)n_1 + (n+k-1)n_2 - n_4, ~\text{for}~ 1 \leq k \leq n \\[1mm]
(r-k+2)n_1 + (k-1)n_2 + n_3-n_4, ~\text{for}~ 1 \leq k \leq n \\[1mm]
(r-n+1)n_1 + nn_2 + n_3-n_4, (n-2)n_1 + nn_2 + 2n_3-n_4, \\[1mm]
(r-2n+2)n_1+ 2nn_2-n_4
\end{array}
\right\rbrace.
\end{align*}
\end{proposition}

\begin{proposition}\label{minmaxback}
Let $S = \langle n_1,n_2,n_3,n_4 \rangle$ be a numerical semigroup generated by the integers $n_1,n_2,n_3,n_4$  defined in Backelin's curves. Then $k\llbracket S \rrbracket$ never has maximal or minimal reduced type.
\end{proposition}
\begin{proof}
From Proposition \ref{pfbackelin}, first observe that
{\small{\begin{align*}
\mathrm{F}(S) = \mathrm{max} \left\lbrace
		\begin{array}{c}
			(3n-2)n_3 -n_4, (r-2n+3)n_1 + (2n-1)n_2 - n_4, (r-n+2)n_1 + (n-1)n_2 + n_3-n_4,  \\[1mm]
		(r-n+1)n_1 + nn_2 + n_3-n_4,  (n-2)n_1 + nn_2 + 2n_3-n_4, (r-2n+2)n_1+ 2nn_2-n_4	   
		\end{array}
		\right\rbrace .
		\end{align*}}}
Thus, we get $\mathrm{F}(S) = (r-n+1)n_1 + nn_2 + n_3-n_4$. Therefore, we have 
$$\mathrm{F}(S) - n_1 + 1 = (r-n)n_1 + nn_2 \geq (2n+2)n_1 + nn_2 > (n-2)n_1+(n+4)n_1 + nn_2.$$
Since $(n+4)n_1 + nn_2 > 3n_3-1$, we get $\mathrm{F}(S) - n_1 + 1 > (n-2)n_1 + 4n_3-n_4 \geq \mathrm{min}~\mathrm{PF}(S)$. Hence,  $k\llbracket S \rrbracket$ does not have maximal reduced type. Now, note that $(2-n)n_1 + nn_2 -n_4 = n_1-2 > 0$. Thus we get,
$$\mathrm{F}(S) - n_1 + 1 = (r-n)n_1+nn_2 < (r-n)n_1+nn_2+(2-n)n_1 + nn_2 -n_4 =  (r-2n+2)n_1+ 2nn_2-n_4.$$
Since, $(r-2n+2)n_1+ 2nn_2-n_4 \in \mathrm{PF}'(S)$, we get $\mathrm{max}~\mathrm{PF}'(S) \geq (r-2n+2)n_1+ 2nn_2-n_4$. This implies that $\mathrm{F}(S) - n_1 + 1 < \mathrm{max}~\mathrm{PF}'(S)$. Hence,  $k\llbracket S \rrbracket$ does not have minimal reduced type.
\end{proof}

\begin{example}
\rm{Let $n = 2$, $ r = 8$. Then $n_1 = 67$, $n_2 = 70$, $n_3 = 74$, $n_4 = 75$, and $S = \langle 67,70,74,75 \rangle$. By Proposition \ref{pfbackelin}, we have $\mathrm{PF}(S) = \{213, 221, 601, 602, 604, 605, 607, 608\}$. Therefore, we get $\mathrm{type}(k\llbracket  S \rrbracket) = 8$, and $s(k\llbracket  S \rrbracket) = 6$. Thus, $k\llbracket S \rrbracket$ has neither maximal nor minimal reduced type.}
\end{example}

\subsection{Bresinsky's curves}
Let $S$ be a numerical semigroup minimally generated by $n_1 \ldots, n_e$. Let $A = k[x_1, \ldots, x_e]$ be a graded polynomial ring with $\mathrm{deg}(x_i) = n_i$ for $1 \leq i \leq e$, and $\phi: k[x_1, \ldots, x_e] \rightarrow k[t]$ be a $k$-algebra map defined by $\phi(x_i) = t^{n_i}$ for $1 \leq i \leq e$. Let $I_S$ be the kernel of the map $\phi$. Note that $I_S$ is a graded ideal, known as the defining ideal of $k[S]$, and $k[S]:= k[t^{n_1}, \ldots, t^{n_e}] \cong \frac{k[x_1, \ldots, x_e]}{I_S}$ is a graded one-dimensional $k$-algebra associated to $S$. In \cite{bresinsky1}, Bresinsky gave an example of monomial curves in the affine $4$-space $\mathbb{A}^4$ defined by the integers $n_1 = 2h(2h+1)$, $n_2 = (2h-1)(2h+1)$, $n_3 = 2h(2h+1)+(2h-1)$, and $n_4 = (2h-1)2h$, where $h \geq 2$. With the help of this example, he showed that the defining ideal of a monomial curve in the affine $n$-space $\mathbb{A}^n$ does not have an upper bound on the minimal generators. Let $S = \langle n_1, n_2, n_3, n_4 \rangle$ be a numerical semigroup generated by the integers $n_1, n_2, n_3, n_4$ defined in Bresinsky's curves.  Bresinsky \cite{bresinsky1} proved that the set $\mathcal{A} = A_1 \cup A_2 \cup \{g_1, g_2 \}$ generates the defining ideal $I_S$, where $A_1 = \{f_{\mu} \mid f_{\mu} = x_1^{\mu-1}x_3^{2h-\mu} - x_2^{2h-\mu}x_4^{\mu+1}, ~ 1 \leq \mu \leq 2h\}$, $A_2 = \{f \mid f = x_1^{\nu_1}x_4^{\nu_4} - x_2^{\mu_2}x_3^{\mu_3}, ~ \nu_1, \mu_3 < 2h-1\}$, $g_1 = x_1^{2h-1}-x_2^{2h}$, and $g_2 = x_3x_4 - x_1x_2$. However, this set is not a minimal generating set. In \cite{ranjana}, the authors extracted a minimal generating set $B = A_1 \cup A_2' \cup \{g_1,g_2\}$ from $\mathcal{A}$ and gave a minimal free resolution of $k[S]$, where $A_2' = \{h_{\mu} \mid h_{\mu} = x_1^{\mu}x_4^{2h+1-\mu} - x_2^{2h - \mu}x_3^{\mu}, ~ 1 \leq \mu \leq 2h-2\}$. In the following theorem, we give a description of pseudo-Frobenius elements of the numerical semigroups associated to Bresinsky's curves.

\begin{theorem}\label{pfbresinsky}
Let $S = \langle n_1,n_2,n_3,n_4 \rangle$ be a numerical semigroup generated by the integers $n_1,n_2,n_3,n_4$ defined in Bresinsky's curves. Then 
\begin{align*}
\mathrm{PF}(S) &= \left\lbrace
		\begin{array}{c}
			 (2h-1)^3 + 4h(h-2) + k(2h-1) + 1, ~\text{for}~ 0 \leq k \leq 2h-3   \\[1mm]
			 (2h-1)^3 + 4h(h-2) + 2h(2k+1) + 2, ~\text{for}~ 0 \leq k \leq 2h-2
		\end{array}
		\right\rbrace .
\end{align*}
\end{theorem}
\begin{proof}
From \cite{ranjana}, we have
\begin{align*}
		I_S = \Bigg\langle 
		\begin{array}{c}
		 x_1^{2h-1}-x_2^{2h},x_3x_4 - x_1x_2,
x_1^{\mu-1}x_3^{2h-\mu} - x_2^{2h-\mu}x_4^{\mu+1} \mid ~ 1 \leq \mu \leq 2h, \\[1mm] 
x_1^{\mu}x_4^{2h+1-\mu} - x_2^{2h - \mu}x_3^{\mu} \mid ~ 1 \leq \mu \leq 2h-2	
		\end{array}
		\Bigg\rangle.
	\end{align*}
Consequently, observe that 
\begin{align*}
		I_S + \langle x_4 \rangle = \Bigg\langle 
		\begin{array}{c}
		x_4, x_1x_2, x_1^{2h-1}, x_3^{2h-1}, x_2^{2h},x_1^{\mu-1}x_3^{2h-\mu} \mid 2 \leq \mu \leq 2h, \\[1mm] 
x_3^{\mu}x_2^{2h-\mu} \mid 1 \leq \mu \leq 2h-2
		\end{array}
		\Bigg\rangle.
	\end{align*}
	Hence, the image of 
	\[ \mathcal{B} = \left\lbrace 
	\begin{array}{c}
	\bigcup_{i=1}^{2h-3} \{x_1^ix_3^j \mid 1 \leq j \leq 2h-2-i \} \cup \bigcup_{i=1}^{2h-2} \{x_3^ix_2^j \mid 1 \leq j \leq 2h-1-i \} \\[1mm] 
	\cup \{x_1^j \mid 0 \leq j \leq 2h-2 \} 
	\cup \{x_2^j \mid 1 \leq j \leq 2h-1 \} \cup \{x_3^j \mid 1 \leq j \leq 2h-2 \} 
	\end{array}
	\right\rbrace \]
	gives a $k$-basis of $k[S]/\langle x_4 \rangle = k[x_1,\ldots,x_4]/(I_S+\langle x_4 \rangle).$ Since $\mathrm{Ap}(S,n_4) = \{ \deg u \mid u \in \mathcal{B} \},$ we get
	\begin{align*}
		&\mathrm{Ap}(S,n_4) = \left\lbrace
		\begin{array}{c}
			 \ \bigcup_{i=1}^{2h-3} \{i n_1 + j n_3 \mid 1 \leq j \leq 2h-2-i \}  \cup \bigcup_{i=1}^{2h-2} \{i n_3 + jn_2 \mid 1 \leq j \leq 2h-1-i \} \\[1mm]
			\cup \{a n_1,  b n_2,  c n_3 \mid 1 \leq j \leq 2h-2, \ 1 \leq b \leq 2h-1, \ 1 \leq c \leq 2h-2 \} \cup \{0\}
		\end{array}
		\right\rbrace.
	\end{align*} 
Let $\preceq_S$ denote the partial order on $\mathbb{N}$ where for all elements $a_1, a_2 \in \mathbb{N}$, $a_1 \preceq_S a_2$ if $a_2-a_1 \in S.$ Thus, we get that
	\begin{align*}
		&\mathrm{max}_{\preceq_S} (\mathrm{Ap}(S,n_4) = \left\lbrace
		\begin{array}{c}
			 \{ (2h-2)n_1, (2h-1)n_2 \} \cup \{ i n_1 + (2h-2-i)n_3 \mid 1 \leq i \leq 2h-3 \}  \\[1mm]
			\cup \{ i n_2 + (2h-1-i)n_3 \mid 1 \leq i \leq 2h-2 \}  
		\end{array} 
		\right\rbrace  .
	\end{align*}
Now, by \cite[Proposition 2.20]{numerical}, we conclude that
\begin{align*}
		&\mathrm{PF}(S) =  \left\lbrace
		\begin{array}{c}
			 \{ (2h-2)n_1 - n_4, (2h-1)n_2-n_4 \} \cup \{ i n_1 + (2h-2-i)n_3 -n_4 \mid 1 \leq i \leq 2h-3 \}  \\[1mm]
			\cup \{ i n_2 + (2h-1-i)n_3 -n_4 \mid 1 \leq i \leq 2h-2 \}  
		\end{array} 
		\right\rbrace  .
	\end{align*}
Now, one can easily write $\mathrm{PF}(S)$ in the required form.	
\end{proof}

%
%

\begin{example}\label{h=2}
\rm{Let $h = 2$, then $n_1 = 20$, $n_2 = 15$, $n_3 = 23$, $n_4 = 12$, and $S = \langle 12, 15, 20, 23 \rangle$. By above theorem $\mathrm{PF}(S) = \left\lbrace	28,31,33,41,49	\right\rbrace .$}
\end{example}

\begin{proposition}\label{minmaxbres}
Let $S = \langle n_1,n_2,n_3,n_4 \rangle$ be a numerical semigroup generated by $n_1,n_2,n_3,n_4$ as defined in Bresinsky's curves. Then $k\llbracket S \rrbracket$ never has maximal or minimal reduced type.
\end{proposition}

\begin{proof}
From Theorem \ref{pfbresinsky}, note that $\mathrm{F}(S) = (2h-1)^3+ 4h(h-2)+ 2h(4h-3)+2$ and $\mathrm{min}~\mathrm{PF}(S) = (2h-1)^3+ 4h(h-2)+1$. Thus, we have
\begin{align*}
 \mathrm{F}(S) - \mathrm{min}~\mathrm{PF}(S)  & = 2h(4h-3) + 1 > 2h(2h-1) - 1. 
\end{align*}
Since, $m(S) = 2h(2h-1)$, we get $\mathrm{min}~\mathrm{PF}(S) < \mathrm{F}(S) - m(S) + 1$. Hence,  $k\llbracket S \rrbracket$ does not have maximal reduced type. Now, note from Theorem \ref{pfbresinsky} that $\mathrm{F}(S) - \mathrm{max}~\mathrm{PF}'(S) =  4h$. Thus,
\begin{align*}
 \mathrm{F}(S) - m(S) + 1 - \mathrm{max}~\mathrm{PF}'(S)  & = 4h(1-h)+ 2h + 1. 
\end{align*}
Since $h \geq 2$, we get $\mathrm{max}~\mathrm{PF}'(S) >  \mathrm{F}(S) - m(S) + 1$. Hence,  $k\llbracket S \rrbracket$ does not have minimal reduced type.
\end{proof}

\begin{example}
\rm{Let $S$ be as in Example \ref{h=2}. Since $\mathrm{PF}(S) = \{28,31,33,41,49\}$, we get $\mathrm{type}(k\llbracket  S \rrbracket) = 5$ and the reduced type $s(k\llbracket  S \rrbracket) = 2$, . Thus, $k\llbracket S \rrbracket$ has neither maximal nor minimal reduced type.}
\end{example}

\section{Gluing of numerical semigroups}

In this section, we study the extremal behavior of the reduced type of semigroup rings associated with the nice extension of a numerical semigroup. In \cite[Theorem 3.34]{mukundan}, the authors prove that if a numerical semigroup $S$ is a gluing of numerical semigroups $S_1$ and $S_2$ then $s(k\llbracket S \rrbracket) \leq s(k\llbracket S_1 \rrbracket) s(k\llbracket S_1 \rrbracket)$, and consequently prove that if $k\llbracket S_1 \rrbracket$ and $k\llbracket S_2 \rrbracket$ have minimal reduced type so has $k\llbracket S \rrbracket$. However, such a nice analog for maximal reduced type does not hold. In this section, we give a sufficient condition for this analog for a special kind of gluing of numerical semigroups. Let us first recall the definition of gluing of numerical semigroups.

\begin{definition}  \label{NSgluing}
	Let $S_1$ and $S_2$ be two numerical semigroups minimally generated by $n_1,\ldots,n_r$ and $n_{r+1},\ldots,n_e$ respectively. Let $\mu \in S_1 \setminus \{ n_1,\ldots,n_r \}$ and $\lambda \in S_2 \setminus \{ n_{r+1}, \ldots, n_e \}$ be such that $\gcd(\lambda,\mu)=1.$ We say that $S = \langle \lambda n_1, \ldots, \lambda n_r, \mu n_{r+1}, \ldots, \mu n_{e} \rangle$ is a gluing of $S_1$ and $S_2,$ denoted by $S_1+_{\lambda \mu} S_2$.
\end{definition}

\begin{coro}[{\cite[Proposition 6.6]{nari}}]\label{pfgluing}
	If $S = \langle \mu S_1, \lambda S_2 \rangle$ (as in Definition {\rm \ref{NSgluing}}), then
	\begin{enumerate}
\item[(i)]	$ \mathrm{PF}(S) = \{ \lambda f + \mu g + \lambda \mu \mid f \in \mathrm{PF}(S_1), g \in \mathrm{PF}(S_2) \}.$
	
	\item[(ii)] $\mathrm{type}(k[S]) = \mathrm{type}(k[S_1])\mathrm{type}(k[S_2])$.
\item[(iii)] $\mathrm{F}(S) = \mu \mathrm{F}(S_1) + \lambda \mathrm{F}(S_1) + \lambda \mu$.
	\end{enumerate}
\end{coro}

\begin{proposition}\label{gluingmaxredtype}
Let $k\llbracket S_1 \rrbracket$ and $k\llbracket S_2 \rrbracket$ be numerical semigroup rings with maximal reduced type, and $S = S_1 +_{\lambda \mu}S_2$ be a gluing of $S_1$ and $S_2$. If $\lambda + \mu > \mathrm{max}\{\lambda n_1, \mu n_{r+1}\}$, then $k\llbracket S \rrbracket$ has maximal reduced type.
\end{proposition}
\begin{proof}
Since $k\llbracket S_1 \rrbracket$ and $k\llbracket S_2 \rrbracket$ have maximal reduced type, we have $\mathrm{min}~\mathrm{PF}(S_1) \geq \mathrm{F}(S_1)-n_1+1$ and $\mathrm{min}~\mathrm{PF}(S_2) \geq \mathrm{F}(S_2)-n_{r+1}+1$. Therefore, we get
$$\lambda \mathrm{min}~\mathrm{PF}(S_1) + \mu \mathrm{min}~\mathrm{PF}(S_2)\geq \lambda \mathrm{F}(S_1)-\lambda n_1+\lambda + \mu \mathrm{F}(S_2)-\mu n_{r+1}+\mu.$$
Now, suppose $\lambda n_1 < \mu n_{r+1}$, then $m(S) = \lambda n_1$, and we get 
$$\lambda \mathrm{min}~\mathrm{PF}(S_1) + \mu \mathrm{min}~\mathrm{PF}(S_2) + \lambda \mu \geq \lambda \mathrm{F}(S_1) + \mu \mathrm{F}(S_2)+ \lambda \mu - m(S)+\lambda +\mu -\mu n_{r+1}.$$ Since $\lambda + \mu > \mathrm{max}\{\lambda n_1, \mu n_{r+1}\}$, by Corollary \ref{pfgluing}, we get $\mathrm{min}~\mathrm{PF}(S) \geq \mathrm{F}(S)-m(S)+1$. Thus,  we get that $k\llbracket S \rrbracket$ has maximal reduced type. Now, if $\lambda n_1 > \mu n_{r+1}$, then $m(S) = \mu n_{r+1}$. Thus, we get
$$\lambda \mathrm{min}~\mathrm{PF}(S_1) + \mu \mathrm{min}~\mathrm{PF}(S_2) + \lambda \mu \geq \lambda \mathrm{F}(S_1) + \mu \mathrm{F}(S_2)+ \lambda \mu - m(S)+\lambda +\mu -\lambda n_1.$$
Since $\lambda + \mu > \mathrm{max}\{\lambda n_1, \mu n_{r+1}\}$,  we get that $k\llbracket S \rrbracket$ has maximal reduced type.
\end{proof}

Observe that the converse of the above proposition is not true, i.e., if  $\lambda + \mu \leq \mathrm{max}\{\lambda n_1, \mu n_{r+1}\}$ and $k\llbracket S_1 \rrbracket$, $k\llbracket S_2 \rrbracket$ have maximal reduced type then $k\llbracket S \rrbracket$ may have maximal reduced type.

\begin{example}
\rm{Let $S = \langle 35, 42, 49, 26 \rangle$. Note that $S$ is a gluing of $S_1 = \langle 5,6,7\rangle$ and $S_2 = \mathbb{N}$ with $\lambda = 7$ and $\mu = 26$. Here $\lambda + \mu = 33 < \mathrm{max}~\{26,35\}$ and $k\llbracket S_1 \rrbracket$, $k\llbracket S_2 \rrbracket$ have maximal reduced type. Note that $\mathrm{PF}(S) = \{212, 219\}$. Therefore, we get $\mathrm{F}(S)-m(S)+1 = 194 < 212$. Thus, $k\llbracket S \rrbracket$ has maximal reduced type.}
\end{example}

\begin{remark}
\rm{Note that, if $\lambda + \mu > \mathrm{max}\{\lambda n_1, \mu n_{r+1}\}$ then either $S_1$ or $S_2$ has to be $\mathbb{N}$. Because, if $\lambda+\mu > \lambda n_1$ and $\lambda+\mu > \mu n_{r+1}$ then we get $\lambda > (n_1'-1)\mu$ and $\mu > (n_1-1)\lambda$. This implies that $\mu > (n_1-1)(n_{r+1}-1)\mu$. This is true only if $n_1=1$ or $n_{r+1} = 1$. However, this condition gives a bigger class of numerical semigroups than the nice extension, first defined in \cite{arslan} as follows.}
\end{remark}

\begin{definition}[{\cite[Definition 3.1]{arslan}}]
Let $S = \langle n_1, \ldots, n_e \rangle$ be a numerical semigroup. A numerical semigroup $S'$ is called a nice extension of $S$ if there is an element $n_{e+1} = a_1n_1 + \ldots + a_en_e \in S \setminus \{n_1, \ldots, n_e\}$ and an integer $p \leq a_1 + \ldots + a_e$ with $\mathrm{gcd}(p,n_{e+1}) = 1$ such that $S' = \langle pn_1, \ldots, pn_k, n_{e+1}\rangle$. 
\end{definition}

\begin{coro}\label{maxniceext}
Let $S = \langle n_1, \ldots, n_e \rangle$ be a numerical semigroup minimally generated by $n_1, \ldots, n_e$ and $S'$ be a nice extension of $S$. Then $k\llbracket S \rrbracket$ has maximal reduced type if and only if $k\llbracket S' \rrbracket$ has maximal reduced type.
\end{coro}
\begin{proof}
Suppose $k\llbracket S \rrbracket$ has maximal reduced type. Let $S' = \langle pn_1, \ldots, pn_k, n_{e+1}\rangle$ be a nice extension of $S$. Taking $\lambda = p$, $\mu = n_{e+1}$, we get $S' = S +_{\lambda \mu} \mathbb{N}$. Since $\lambda + \mu = p + n_{e+1} > \mathrm{max}\{pn_1, n_{e+1}\} = n_{e+1}$, $k\llbracket S' \rrbracket$ has maximal reduced type follows from Proposition \ref{gluingmaxredtype}. Now, suppose $k\llbracket S' \rrbracket$ has maximal reduced type. Since $S' = S +_{\lambda \mu} \mathbb{N}$, $S$ has maximal reduced type follows from \cite[Theorem 3.40]{mukundan}.
\end{proof}

\begin{example}
\rm{Let $S = \langle 25, 30, 35, 23 \rangle$. Note that $S$ is a gluing of $S_1 = \langle 5,6,7\rangle$ and $S_2 = \mathbb{N}$ with $\lambda = 5$ and $\mu = 23$. Now $\lambda + \mu = 28 > \mathrm{max}~\{23,25\}$ and $k\llbracket S_1 \rrbracket$ has maximal reduced type. Thus, by Proposition \ref{gluingmaxredtype}, $k\llbracket S \rrbracket$ has maximal reduced type. One can easily observe that $S$ is not a nice extension of $S_1$.}
\end{example}

\section{The duplication of a numerical semigroup}

Let $S$ be a numerical semigroup and $E \subseteq S$ be an ideal of $S$, i.e., $S + E \subseteq E$. For an odd integer $d \in S$ and an ideal $E \subseteq S$, the  numerical duplication (see \cite{d'anna}) of $S$ with respect to $E$ and $d$ is defined as $$ S \bowtie^d E = 2 \cdot S \cup (2\cdot E + d),$$  
\noindent
where $2 \cdot S = \{2s \mid s \in S\}$ and $2 \cdot E + d = \{2e + d \mid e \in E\}$.  
Observe that the set $S \bowtie^d E$ is a numerical semigroup. Also, note that $S$ is an ideal of $S$. We set $E = S$ and call the set $S \bowtie^d S$ a self-duplication of $S$ with respect to $d$.

\begin{remark}
\rm{For an ideal $E \subseteq S$, $\tilde{E}:= E\cup \{0\}$ is also a numerical semigroup. Because, since $S+E \subseteq E$, we have $x+y \in E$ for all $x,y \in E$. Let $m$ be the smallest element of $E$. Note that $m + \mathrm{F}(S)+1+n \in E$ for all $n \in \mathbb{N}$. Thus $\mathbb{N} \setminus \tilde{E}$ is finite.}
\end{remark}

In the following theorem, we give a description of pseudo-Frobenius elements of duplication of a numerical semigroup.

\begin{theorem}\label{pfduplication}
 Let $E$ be an ideal of a numerical semigroup $S$, and $S \bowtie^d E$ be a duplication of $S$. Then we have the following:
\begin{enumerate}
\item[(i)] If $E \neq S,S^* (S^* := S \setminus \{0\})$, then $ \mathrm{PF}(S \bowtie^d E) = \Delta_1 \cup \Delta_2 $, where
\[ \Delta_1 = \{2f \mid f \in \mathrm{PF}(S) \cap \mathrm{PF}(\tilde{E})\}, \]
and
\[ \Delta_2 = \{2f + d \mid f \in \mathrm{PF}(\tilde{E}) ~\text{and}~ f+s \in E ~\text{for}~ s \in S\setminus \tilde{E}  \}.\]

\item[(ii)] If $E = S^*$, then 
\[\mathrm{PF}(S \bowtie^d E) = \{d,2f,2f + d \mid f \in \mathrm{PF}(S)\}. \]

\item[(iii)] If $E = S$, then 
\[\mathrm{PF}(S \bowtie^d E) = \{2f + d \mid f \in \mathrm{PF}(S)\}. \]
\end{enumerate} 
Moreover, we have 
\begin{itemize}
\item[(i)] If $E \neq S,S^*$, the Cohen-Macaulay type of $k\llbracket S \bowtie^d E \rrbracket$ is $\vert \Delta_1 \cup \Delta_2 \vert.$
\item[(ii)] If $E = S^*$, the Cohen-Macaulay type of $k\llbracket S \bowtie^d E \rrbracket$ is $2 \mathrm{type} (k\llbracket S \rrbracket) + 1.$
\item[(iii)] If $E = S$, the Cohen-Macaulay type of $k\llbracket S \bowtie^d E \rrbracket$ is equal to the Cohen-Macaulay type of $k\llbracket S \rrbracket.$
\end{itemize}

\end{theorem}
\begin{proof}
(i) Let $f \in \mathrm{PF}(S)$, then clearly $2f \notin S \bowtie^d E$. Now, let $\mathfrak{f} \in \Delta_1$ and $0 \neq s \in S$, then we have
$\mathfrak{f} + 2s = 2(f+s)$ for some $f \in \mathrm{PF}(S) \cap \mathrm{PF}(\tilde{E})$. Since $f \in \mathrm{PF}(S)$, we get $\mathfrak{f} + 2s = 2s'$ for some $s' \in S$. For $e \in E$, we have   
$\mathfrak{f} + 2e + d = 2(f+e)+d$ for some $ f \in \mathrm{PF}(S) \cap \mathrm{PF}(\tilde{E})$. Since $f \in \mathrm{PF}(\tilde{E})$, this implies $f + e \in E$. Therefore, we get $\mathfrak{f} + 2e + d = 2e' + d$ for some $e' \in E$. Thus, we get $\Delta_1 \subseteq \mathrm{PF}(S \bowtie^d E).$ Now, let $f \in \mathrm{PF}(\tilde{E})$, then clearly $2f + d \notin S \bowtie^d E$.
Let $\mathfrak{f} \in \Delta_2$ and $0 \neq s \in S$, then we have $\mathfrak{f} + 2s = 2(f+s)+ d$ for some $f \in \mathrm{PF}(\tilde{E})$ and  $f+s \in E$ for $s \in S \setminus \tilde{E}$. This implies that $\mathfrak{f}+2s \in 2\cdot E + d$. For $e \in E$, we have $\mathfrak{f} + 2e + d = 2f + d + 2e + d = 2(f+e+d)$ for some $f \in \mathrm{PF}(\tilde{E})$. Therefore, $f + e \in E$ and hence $\mathfrak{f} + 2e + d \in 2 \cdot S$. Thus, we get $\Delta_2 \subseteq \mathrm{PF}(S \bowtie^d E)$. Now, consider $\mathfrak{f} \in \mathrm{PF}(S \bowtie^d E)$. Suppose $\mathfrak{f}$ is even then $\mathfrak{f} = 2f$ for some $f \in \mathbb{Z} \setminus S$. Therefore, we have $2f + 2s = 2(f+s) \in S \bowtie^d E$ for all $0 \neq s \in S$ and $2f + 2e + d = 2(f+e) + b \in S \bowtie^d E $ for all $e \in E$. Therefore, we get $f + s \in S$ for all $0 \neq s \in S$ and $ f + e \in E$ for all $e \in E$. Thus, we get $f \in \mathrm{PF}(S) \cap \mathrm{PF}(\tilde{E})$. Now suppose $\mathfrak{f}$ is odd. Since $d$ is odd, we can write $\mathfrak{f} = 2f + d$ for some $f \in \mathbb{Z} \setminus E$. Therefore, we have $2f + d + 2s = 2(f+s) + d \in S \bowtie^d E$ for all $0 \neq s \in S$. Suppose $f = 0$, then $d+2s \in S \bowtie^d E$ for all $0 \neq s \in S$. Therefore, $2s \in 2\cdot E$ for all $0 \neq s \in S$. This implies $E = S^*$, which is not the case. Therefore, we may assume $f \neq 0$. Thus, we get $f + s \in E$ for all $0 \neq s \in S$. Hence, $f \in \mathrm{PF}(\tilde{E})$ and $f+s \in E$ for $s \in S\setminus \tilde{E}$.

(ii) Follows from the proof of part (i).

(iii) Similar to part (i), it is easy to see that $\{2f + d \mid f \in \mathrm{PF}(S)\} \subseteq \mathrm{PF}(S \bowtie^d S).$ Now assume $\mathfrak{f} \in \mathrm{PF}(S \bowtie^d S)$, then $f$ is not an even integer. Suppose $f$ is even then $f = 2z $ for some $z \in \mathbb{Z}\setminus S$, and $\mathfrak{f} + 2s +d \in S \bowtie^d S$ for all $s \in S$. Thus, $z + s \in S$ for all $s \in S$. Now, taking $s = 0$, we get a contradiction. Thus, $\mathfrak{f}$ has to be odd. Write $ \mathfrak{f} = 2z+d$ for some $z \in \mathbb{Z} \setminus S$. Therefore, for all $0 \neq s \in S$, we have $\mathfrak{f} + 2s \in S \bowtie^d S$. Thus, we get $z + s \in S$ for all $0 \neq s \in S$, and hence $z \in \mathrm{PF}(S)$. Therefore, we get $\mathrm{PF}(S \bowtie^d S) \subseteq \{2f + d \mid f \in \mathrm{PF}(S)\}$.
\end{proof}

\begin{example}
\begin{itemize}
\rm{\item[(i)] Let $S = \langle 3,4,5 \rangle$, $E = \{5+i \mid i \in \mathbb{N}\}$ and $d = 11$. Note that $E \subseteq S$ is an ideal of $S$. Therefore, $\tilde{E} = \langle 5,6,7,8,9 \rangle$ and $\mathrm{PF}(\tilde{E}) = \{1,2,3,4\}$. Since $\mathrm{PF}(S) = \{1,2\}$, by Theorem \ref{pfduplication} part (i), we get $\mathrm{PF}(S \bowtie^{11} E) = \{2,4,15,17,19\}$.

\item[(ii)] Let $S = \langle 3,4,5 \rangle$, $E = S^*$ and $d = 11$. Then, $\mathrm{PF}(S) = \{1,2\}$. Therefore, by Theorem \ref{pfduplication} part (ii), we get $\mathrm{PF}(S \bowtie^{11} E) = \{2,4,11,13,15\}.$

\item[(iii)] Let $S = \langle 3,4,5 \rangle$, $E = S$ and $d = 11$. Then, $\mathrm{PF}(S) = \{1,2\}$. Therefore, by Theorem \ref{pfduplication} part (iii), we get $\mathrm{PF}(S \bowtie^{11} E) = \{13,15\}.$}
\end{itemize}

\end{example}

\begin{remark}
For a fixed $r > 0$, there are infinitely many numerical semigroups with type $r$. Let $S = \langle r+1, r+2, \ldots, 2r+1 \rangle$. Then $\mathrm{PF}(S) =\{1,2, \ldots, r\} $, and thus $\mathrm{type}(k\llbracket S \rrbracket) = r$. Take any odd natural number $d \in S$, by Theorem \ref{pfduplication} part (iii), we get $\mathrm{type}(k\llbracket S \bowtie^d S \rrbracket) = r.$ Since there are infintely many choices for $d$, we get an infinite class of numerical semigroups with type $r$.
\end{remark}

\begin{theorem}\label{minreddup}
Let $S$ be a numerical semigroup and $E \subseteq S$ be an ideal of $S$. Then, we have the following:
\begin{itemize}
\item[(i)] Let $E = S$. If $k\llbracket S \rrbracket$ has minimal reduced type, then $k\llbracket S \bowtie^d E \rrbracket$ has minimal reduced type.

\item[(ii)] Let $E = S^*$, $k\llbracket S \bowtie^d E \rrbracket$ is never Gorenstein. 
\begin{enumerate}

\item[(a)] Suppose $k\llbracket S \rrbracket$ is Gorenstein, 
\begin{enumerate}
\item[(a.1)] when $d < 2 \mathrm{F}(S)$, $k\llbracket S \bowtie^d E \rrbracket$ has minimal reduced type if and only if $2 m(S) < d.$
\item[(a.2)] when $d>2\mathrm{F}(S)$, $k\llbracket S \bowtie^d E \rrbracket$ has minimal reduced type if and only if $m(S) < \mathrm{F}(S)+1.$
\end{enumerate}

\item[(b)] Suppose $k\llbracket S \rrbracket$ is not Gorenstein. 
\begin{enumerate}
\item[(b.1)] For $\mathrm{max}~\mathrm{PF}'(S \bowtie^d E) \neq 2\mathrm{F}(S)$, if $k\llbracket S \rrbracket$ has minimal reduced type then $k\llbracket S \bowtie^d E \rrbracket$ also has minimal reduced type.
\item[(b.2)] For $\mathrm{max}~\mathrm{PF}'(S \bowtie^d E) = 2\mathrm{F}(S)$, $k\llbracket S \bowtie^d E \rrbracket$ has minimal reduced type if and only if $d > 2m(S)$.
\end{enumerate}

\end{enumerate}
\item[(iii)] Let $E \neq S,S^*$,
\begin{enumerate}
\item[(a)] Suppose $\mathrm{F}(S) \neq \mathrm{F}(\tilde{E})$. If $k\llbracket \tilde{E} \rrbracket$ has minimal reduced type then $k\llbracket S \bowtie^d E \rrbracket$ also has minimal reduced type.
\item[(b)] Suppose $\mathrm{F}(S) = \mathrm{F}(\tilde{E})$,
\begin{enumerate}
\item[(b.1)] For $\mathrm{max}~\mathrm{PF}'(S \bowtie^d E) \neq 2\mathrm{F}(S)$, if $k\llbracket \tilde{E} \rrbracket$ has minimal reduced type then $k\llbracket S \bowtie^d E \rrbracket$ also has minimal reduced type.
\item[(b.2)] For $\mathrm{max}~\mathrm{PF}'(S \bowtie^d E) = 2\mathrm{F}(S)$, $k\llbracket S \bowtie^d E \rrbracket$ has minimal reduced type if and only if $d > 2m(S)$.
\end{enumerate}
\end{enumerate}
\end{itemize}
\end{theorem}

\begin{proof}
(i) Suppose $k\llbracket S \rrbracket$ has minimal reduced type. This implies that $\mathrm{max}~\mathrm{PF}'(S) < \mathrm{F}(S) - m(S) + 1$. From Theorem \ref{pfduplication}, note that $\mathrm{F}(S \bowtie^d E) = 2 \mathrm{F}(\tilde{E}) + d$. Therefore, we get
\[\mathrm{max}~\mathrm{PF}'(S \bowtie^d S) = 2 ~ \mathrm{max}~\mathrm{PF}'(S)+d < 2 \mathrm{F}(S)+d-2m(S)+ 1 = \mathrm{F}(S \bowtie^d S)-2m(S) + 1. \]
If $d > 2 m(S)$, the multiplicity of $S \bowtie^d S$ is $2m(S)$, and hence  $k\llbracket S \bowtie^d S \rrbracket$ has minimal reduced type. If $d < 2 m(S)$, $\mathrm{F}(S \bowtie^d S)-2m(S) + 1 < \mathrm{F}(S \bowtie^d S)-d + 1.$ Since $d < 2 m(S)$ the multiplicity of $S \bowtie^d S$ is $d$, and hence  $k\llbracket S \bowtie^d S \rrbracket$ has minimal reduced type.

(ii) Since $E = S^*$, $k\llbracket S \bowtie^d E \rrbracket$ is never Gorenstein follows from Theorem \ref{pfduplication} part (ii). For (a), assume $k\llbracket S \rrbracket$ is Gorenstein, i.e $\mathrm{PF}(S) = \{\mathrm{F}(S)\}$. For (a.1), since $d < 2 \mathrm{F}(S)$, by Theorem \ref{pfduplication} part (ii), we get $\mathrm{max}~\mathrm{PF}'(S \bowtie^d E) = 2 \mathrm{F}(S)$. Since $E = S^*$, the multiplicity of $S \bowtie^d E$ is $2m(S)$. Therefore, $k\llbracket S \bowtie^d E \rrbracket$ has minimal reduced type  if and only if $  2 \mathrm{F}(S) < \mathrm{F}(S \bowtie^d E) - 2m(S) + 1$. Since, $\mathrm{F}(S \bowtie^d E) = 2 \mathrm{F}(\tilde{E}) + d$, $k\llbracket S \bowtie^d E \rrbracket$ has minimal reduced type  if and only if $2m(S) < d+1$. For (a.2), since $d > 2 \mathrm{F}(S)$, by Theorem \ref{pfduplication} part (ii), we get $\mathrm{max}~\mathrm{PF}'(S \bowtie^d E) = d$. Therefore, $k\llbracket S \bowtie^d E \rrbracket$ has minimal reduced type  if and only if $  d < \mathrm{F}(S \bowtie^d E) - 2m(S) + 1$. Since, $\mathrm{F}(S \bowtie^d E) = 2 \mathrm{F}(\tilde{E}) + d$, $k\llbracket S \bowtie^d E \rrbracket$ has minimal reduced type  if and only if $2m(S) < 2\mathrm{F}(S)+1$. For (b), assume $k\llbracket S \rrbracket$ is not Gorenstein i.e $\mathrm{PF}'(S) \neq \emptyset.$ For (b.1), since $k\llbracket S \rrbracket$ has minimal reduced type, we get $\mathrm{max}~\mathrm{PF}'(S) < \mathrm{F}(S) - m(S) + 1$. Therefore,  $2 \mathrm{max}~\mathrm{PF}'(S) + d < 2\mathrm{F}(S) - 2m(S) + 1 + d$. Since $\mathrm{F}(S \bowtie^d E) = 2 \mathrm{F}(\tilde{E}) + d$ and $\mathrm{max}~\mathrm{PF}'(S \bowtie^d E) \neq 2\mathrm{F}(S)$, we get $\mathrm{max}~\mathrm{PF}'(S \bowtie^d E) < \mathrm{F}(S \bowtie^d E) - m(S \bowtie^d E) + 1$. Hence, $k\llbracket S \bowtie^d E \rrbracket$ has minimal reduced type. For (b.2), since $\mathrm{max}~\mathrm{PF}'(S \bowtie^d E) = 2\mathrm{F}(S)$, $k\llbracket S \bowtie^d E \rrbracket$ has minimal reduced type  if and only if $ 2\mathrm{F}(S)  < \mathrm{F}(S \bowtie^d E) - 2m(S) + 1$.  Since, $\mathrm{F}(S \bowtie^d E) = 2 \mathrm{F}(\tilde{E}) + d$, $k\llbracket S \bowtie^d E \rrbracket$ has minimal reduced type  if and only if $2m(S) < d +1$. Since $d$ is odd, $k\llbracket S \bowtie^d E \rrbracket$ has minimal reduced type  if and only if $2m(S) < d$.

(iii) Assume $E \neq S, S^*$. For (a), we have $\mathrm{F}(S) \neq \mathrm{F}(\tilde{E})$. By Theorem \ref{pfduplication} part (ii), either $\mathrm{max}~\mathrm{PF}'(S \bowtie^d E) \in \Delta_1 $ or $\mathrm{max}~\mathrm{PF}'(S \bowtie^d E) \in \Delta_2$. Suppose $\mathrm{max}~\mathrm{PF}'(S \bowtie^d E) \in \Delta_1 $. Since $\mathrm{F}(S) \neq \mathrm{F}(\tilde{E})$, we get $\mathrm{max}~\mathrm{PF}'(S \bowtie^d E) \leq  2 \mathrm{max}~\mathrm{PF}'(\tilde{E})$. Also, since $\tilde{E}$ has minimal reduced type, we have   $\mathrm{max}~\mathrm{PF}'(\tilde{E}) < \mathrm{F}(\tilde{E}) - m(\tilde{E}) + 1.$ This implies $2 \mathrm{max}~\mathrm{PF}'(\tilde{E}) < 2 \mathrm{F}(\tilde{E}) - 2 m(\tilde{E}) + 1.$ Therefore, we get 
$$\mathrm{max}~\mathrm{PF}'(S \bowtie^d E) \leq  2 \mathrm{max}~\mathrm{PF}'(\tilde{E}) < 2 \mathrm{F}(\tilde{E}) - 2 m(\tilde{E}) + 1 < 2 \mathrm{F}(\tilde{E}) - 2 m(\tilde{E}) + 1 + d.$$ 
Since $\mathrm{F}(S \bowtie^d E) = 2 \mathrm{F}(\tilde{E}) + d$ and $2 m(\tilde{E}) \geq 2 m(S) = m(S \bowtie^d E)$, we conclude that $\mathrm{max}~\mathrm{PF}'(S \bowtie^d E) < \mathrm{F}(S \bowtie^d E) - m(S \bowtie^d E) + 1.$ Hence, $k\llbracket S \bowtie^d E \rrbracket$ has minimal reduced type. Now, suppose, $\mathrm{max}~\mathrm{PF}'(S \bowtie^d E) \in \Delta_2 $. This implies that $\mathrm{max}~\mathrm{PF}'(S \bowtie^d E) \leq  2 \mathrm{max}~\mathrm{PF}'(\tilde{E}) + d$. Remaining part follows analogously. For part (b), assume $\mathrm{F}(S) = \mathrm{F}(\tilde{E})$. For (b.1), since $\mathrm{max}~\mathrm{PF}'(S \bowtie^d E) \neq 2\mathrm{F}(S)$, observe that $\mathrm{max}~\mathrm{PF}'(S \bowtie^d E) \leq 2\mathrm{max}~\mathrm{PF}'(\tilde{E}) + d.$ Again, remaining part follows analogously. For (b.2), since $\mathrm{max}~\mathrm{PF}'(S \bowtie^d E) = 2\mathrm{F}(S)$, $k\llbracket S \bowtie^d E \rrbracket$ has minimal reduced type if and only if $ 2\mathrm{F}(S) < \mathrm{F}(S \bowtie^d E) - 2m(S) + 1$. Since, $\mathrm{F}(S \bowtie^d E) = 2 \mathrm{F}(\tilde{E}) + d$ and $\mathrm{F}(S) = \mathrm{F}(\tilde{E})$, $k\llbracket S \bowtie^d E \rrbracket$ has minimal reduced type  if and only if $2m(S) < d + 1$. Since $d$ is odd, $k\llbracket S \bowtie^d E \rrbracket$ has minimal reduced type  if and only if $2m(S) < d$.
\end{proof}

\begin{remark} \rm{For a fixed $r > 0$, there are infinitely many numerical semigroups of minimal reduced type with type $r$. Let $S = \langle r+1, r+1+(r+2), \ldots, r+1 +r(r+2) \rangle$. Since, embedding dimension of $S$ is equal to the multiplicity of $S$, by \cite[Corollary 3.2]{numerical}, we get $\mathrm{type}(S) = r$. Moreover, $\mathrm{PF}(S) = \{r+2, 2(r+2), \ldots, r(r+2)\}$. Therefore, $\mathrm{F}(S) = r(r+2)$, $\mathrm{max}~\mathrm{PF}'(S ) = (r-1)(r+2)$, and $\mathrm{max}~\mathrm{PF}'(S ) < \mathrm{F}(S) - m(S) + 1$. Hence, we obtain that $S$ is of minimal reduced type. Now, take any odd natural number $d \in S$, by Theorem \ref{pfduplication} part (iii), $\mathrm{type}(k\llbracket S \bowtie^d S \rrbracket) = r$, and by Theorem \ref{minreddup} part (i), $k\llbracket S \bowtie^d S \rrbracket$ has minimal reduced type. Since there are infintely many choices for $d$, for each $r > 0$, we get an infinite class of numerical semigroups of minimal reduced type with type $r$.}
\end{remark}

\begin{remark} \rm{In general, even if both $k \llbracket S \rrbracket$ and $k\llbracket \tilde{E} \rrbracket $ have maximal reduced type, $k\llbracket S \bowtie^d E \rrbracket$ may not have maximal reduced type. Because of the difficulty of keeping track of minimum element of the set of pseudo-Frobenius elements of the duplication, it is not feasible to give a result for maximal reduced type in terms of $d$ and $m(S)$ similar to the Theorem \ref{minreddup}.} 
\end{remark}

\begin{example}
 \rm{Let $S = \langle 3,4,5 \rangle$, $E = \{5+i \mid i \in \mathbb{N}\}$ and $d = 11$. Note that $E \subseteq S$ is an ideal of $S$. Note that $\tilde{E} = \langle 5,6,7,8,9 \rangle$ and $\mathrm{PF}(\tilde{E}) = \{1,2,3,4\}$. Observe that both $k \llbracket S \rrbracket$ and $k\llbracket \tilde{E} \rrbracket $ have maximal reduced type. Since $\mathrm{PF}(S) = \{1,2\}$, by Theorem \ref{pfduplication} part (i), we get $\mathrm{PF}(S \bowtie^{11} E) = \{2,4,15,17,19\}.$
Therefore, $\mathrm{min}~\mathrm{PF}(S \bowtie^{11} E) = 2 \ngeq 14 = 19-6+1 = \mathrm{F}(S \bowtie^{11} E) - m(S \bowtie^{11} E) + 1$. Therefore, $k\llbracket S \bowtie^d E \rrbracket$ does not have maximal reduced type.}
\end{example}

In the following results, we explore the maximal reduced type for the duplication of a numerical semigroup when $E = S, S^*$.

\begin{proposition}\label{maxreddup}
Let $S$ be a numerical semigroup. Then, we have the following:
\begin{itemize}
\item[(i)] If $d > 2 m(S)$, then $k\llbracket S \rrbracket$ has maximal reduced type if and only if $k\llbracket S \bowtie^d S \rrbracket$ has maximal reduced type.
\item[(ii)] If $d < 2 m(S)$, then $k\llbracket S \bowtie^d S \rrbracket$ has maximal reduced type if and only if $\frac{d-1}{2} \geq \mathrm{F}(S)-\mathrm{min}~\mathrm{PF}(S)$.
 
\end{itemize}
\end{proposition}

\begin{proof}
(i) Suppose $k\llbracket S \rrbracket$ has maximal reduced type. This implies that $\mathrm{min}~\mathrm{PF}(S) \geq \mathrm{F}(S) - m(S) + 1$. From Theorem \ref{pfduplication}, note that $\mathrm{F}(S \bowtie^d S) = 2 \mathrm{F}(S) + d$. Therefore, we get
\[\mathrm{min}~\mathrm{PF}(S \bowtie^d S) = 2 ~ \mathrm{min}~\mathrm{PF}(S)+d \geq 2 \mathrm{F}(S)+d-2m(S)+1 \geq \mathrm{F}(S \bowtie^d S)-2m(S) + 1.\]
Since $d > 2 m(S)$, the multiplicity of $S \bowtie^d S$ is $2m(S)$, and hence  $k\llbracket S \bowtie^d S \rrbracket$ has maximal reduced type. For the converse part, we will prove that if $k\llbracket S \rrbracket$ does not have maximal reduced type, then $k\llbracket S \bowtie^d S \rrbracket$ does not have maximal reduced type. Suppose $k\llbracket S \rrbracket$ does not have maximal reduced type. This implies that $\mathrm{min}~\mathrm{PF}(S) < \mathrm{F}(S) - m(S) + 1$. Therefore, we get $ 2 ~ \mathrm{min}~\mathrm{PF}(S) + 1 <  2 ~ \mathrm{F}(S) - 2m(S) + 2 $ and thus  $ 2 ~ \mathrm{min}~\mathrm{PF}(S) + d <  2 ~ \mathrm{F}(S) + d - 2m(S) + 1 $. Thus, by Theorem \ref{pfduplication}, we get  $\mathrm{min}~\mathrm{PF}(S \bowtie^d S) < \mathrm{F}(S \bowtie^d S) - 2m(S) + 1$. Since $d > 2m(S)$, the multiplicity of $S \bowtie^d S$ is $2m(S)$, and hence $k\llbracket S \bowtie^d S \rrbracket$ does not have maximal reduced type.\\
\medskip
(ii) If $d < 2m(S)$ then we have $k\llbracket S \bowtie^d S \rrbracket$ has maximal reduced type if and only if $$2 ~ \mathrm{min}~\mathrm{PF}(S)+d \geq 2 \mathrm{F}(S)+1 ~ \text{if and only if}~ \frac{d-1}{2} \geq \mathrm{F}(S)-\mathrm{min}~\mathrm{PF}(S).$$
\end{proof}

\begin{remark} \rm{For a fixed $r > 0$, there are infinitely many numerical semigroups of maximal reduced type with type $r$. Let $S = \langle r+1, r+2, \ldots, 2r+1 \rangle$. Then $\mathrm{PF}(S) =\{1,2, \ldots, r\} $, and thus $\mathrm{type}(k\llbracket S \rrbracket) = r$. It is evident that $k\llbracket S \rrbracket = r$ has maximal reduced type. Now, take any odd natural number $d \in S$ such that $d > 2m(S)$, by Theorem \ref{pfduplication} part (iii), $\mathrm{type}(k\llbracket S \bowtie^d S \rrbracket) = r$, and by Theorem \ref{maxreddup} part (i), $k\llbracket S \bowtie^d S \rrbracket$ has maximal reduced type. Since there are infintely many choices for $d$, for each $r > 0$, we get an infinite class of numerical semigroups of maximal reduced type with type $r$.}
\end{remark}


\begin{example}
\rm{Let $S = \langle 5,6,7\rangle$, and $d = 7$. Note that $\mathrm{PF}(S) = \{8,9\}$. The duplication $S \bowtie^7 S = \langle 7,10,12,14 \rangle$, and $\mathrm{PF}(S \bowtie^7 S) = \{23,25\}$. Thus $\mathrm{min}~\mathrm{PF}(S \bowtie^7 S) = 23 > 19 = \mathrm{F}(S \bowtie^7 S) - m(S\bowtie^7S) + 1$. Thus, $k\llbracket S \bowtie^7 S \rrbracket$ has maximal reduced type.}
\end{example}

\begin{proposition}\label{maxreddup2}
Let $S$ be a numerical semigroup. Then, we have the following:
\begin{itemize}
\item[(i)] If $d < 2~\mathrm{min}~\mathrm{PF}(S)$, then $k\llbracket S \bowtie^d S^* \rrbracket$ has maximal reduced type  if and only if $m(S) > \mathrm{F}(S)$.
\item[(ii)] If $d > 2~\mathrm{min}~\mathrm{PF}(S)$, then $k\llbracket S \bowtie^d S^* \rrbracket$ has maximal reduced type if and only if $m(S) - \mathrm{F}(S) \geq \frac{d+1}{2} - \mathrm{min}~\mathrm{PF}(S)$.
\end{itemize}
\end{proposition}

\begin{proof}
(i) Since, $d < 2~\mathrm{min}~\mathrm{PF}(S)$, by Theorem \ref{pfduplication} part (ii), we get $\mathrm{min}~\mathrm{PF}(S \bowtie^d S^*) = d$. Therefore, $k\llbracket S \bowtie^d S^* \rrbracket$ has maximal reduced type  if and only if $d \geq \mathrm{F}(S \bowtie^d S^*) - 2m(S) + 1$. Since, $\mathrm{F}(S \bowtie^d S^*) = 2 \mathrm{F}(S) + d$, $k\llbracket S \bowtie^d S^* \rrbracket$ has maximal reduced type  if and only if $2m(S) - 1 \geq 2 \mathrm{F}(S)$. Thus, $k\llbracket S \bowtie^d S^* \rrbracket$ has maximal reduced type  if and only if $m(S) > \mathrm{F}(S)$.

(ii) Since, $d > 2~\mathrm{min}~\mathrm{PF}(S)$, by Theorem \ref{pfduplication} part (ii), we get $\mathrm{min}~\mathrm{PF}(S \bowtie^d S^*) = 2~\mathrm{min}~\mathrm{PF}(S)$. Therefore, $k\llbracket S \bowtie^d S^* \rrbracket$ has maximal reduced type  if and only if $ 2~\mathrm{min}~\mathrm{PF}(S) \geq \mathrm{F}(S \bowtie^d S^*) - 2m(S) + 1$. Since, $\mathrm{F}(S \bowtie^d S^*) = 2 \mathrm{F}(S) + d$, $k\llbracket S \bowtie^d S^* \rrbracket$ has maximal reduced type  if and only if $2m(S) - 2 \mathrm{F}(S) \geq (d+1) - 2~\mathrm{min}~\mathrm{PF}(S)$. Hence, the result follows. 
\end{proof}

\end{document}